\documentclass[letterpaper]{article}
\usepackage{proceed2e}
\usepackage[margin=1in]{geometry}

\usepackage{times}
\usepackage{amssymb,amsmath,amsthm,amsfonts,epsf,epsfig,graphicx,mathrsfs,tabularx,color}
\def\X{{\mathbb X}}

\newcommand{\ocap}{\textcircled{$\scriptstyle{\cap}$}}
\newcommand{\ocup}{\textcircled{$\scriptstyle{\cup}$}}
\newcommand{\bi}{\begin{itemize}}
\newcommand{\ei}{\end{itemize}}
\newcommand{\be}{\begin{enumerate}}
\newcommand{\ee}{\end{enumerate}}

\newtheorem{theorem}{Theorem}
\newtheorem{lemma}{Lemma}
\newtheorem{definition}{Definition}

\newtheorem{corollary}{Corollary}	

\newtheorem{conjecture}{Conjecture}

\title{Belief likelihood function 
for generalised logistic regression}

\author{Fabio Cuzzolin \\ School of Engineering, Computing and Mathematics \\
Oxford Brookes University \\ Oxford, UK} 

\begin{document}

\maketitle

\begin{abstract}
The notion of belief likelihood function of repeated trials is introduced, 
whenever the uncertainty for individual trials is encoded by a belief measure (a finite random set).
This generalises the traditional likelihood function, and provides a natural setting for belief inference from statistical data. Factorisation results are proven for the case in which conjunctive or disjunctive combination are employed, leading to analytical expressions for the lower and upper likelihoods of `sharp' samples in the case of Bernoulli trials, and to the formulation of a generalised logistic regression framework. 
\end{abstract}

\section{Introduction} \label{sec:introduction}

\emph{Logistic regression} \cite{Cox58} is a popular statistical method for modelling data in which one or more independent observed variables determine an outcome, represented by a binary variable. The framework can also be extended to the multinomial case, and 
is widely used in various fields, including machine learning, medical diagnosis, and social sciences, to cite a few.
Despite its successes, the method has serious limitations. In particular, it has been shown to consistently and sharply underestimate the probability of `rare' events \cite{King2001}. The term \cite{Falk04} denotes cases in which the training data are of insufficient quality, in the sense that they do not represent well enough the underlying distribution. As a result,
scientists are forced to infer probability distributions
using information captured in `normal' times (e.g. while a nuclear power plant is working nominally), whereas these distributions are later used to extrapolate results at the `tail' of the curve.

Although corrections to logistic regression have been proposed \cite{King2001},
the root cause of the problem, in our view, our very models of uncertainty are themselves affected by uncertainty: a phenomenon often called `Knightian' uncertainty.
The latter can be explicitly modelled by
considering convex sets of probability
distributions, or `credal sets \cite{Levi80,kyburg87bayesian}.
Random sets \cite{matheronrandom,Nguyen78,Shafer87b,Hestir91}, in particular, are a sub-class of credal sets induced by probability distributions on the collection of all subsets of the sample space.
In the finite case random sets are often called \emph{belief functions}, a term introduced by Glenn Shafer \cite{Shafer76} from a subjective probability perspective.

As we show here, the logistic regression framework can indeed be generalised to the case of belief functions, which themselves {generalise} classical discrete probability measures.
Given a sample space $\X$,
the traditional likelihood function 
is equal to the conditional probability of the data given a parameter $\theta \in \Theta$, i.e., a  family of probability distribution functions (PDFs) over $\X$ parameterised by $\theta$:
$L(\theta|X) \doteq p(X|\theta)$, $\theta \in \Theta.$
As originally proposed by Shafer and Wasserman \cite{Shafer76,Wasserman1987some,Wasserman90}, 
belief functions can indeed be built from traditional likelihood functions.
However, as we argue here, one can directly define a \emph{belief likelihood function}, mapping a sample observation $x \in \X$ to a real number, 
as a natural set-valued generalisation of the conventional likelihood.
It is natural to define such a belief likelihood function as family of belief functions on $\X$, $Bel_\X(.|\theta)$, parameterised by $\theta \in \Theta$. 
As the latter take values on \emph{sets} of outcomes, $A \subset \X$, of which singleton outcomes are mere special cases, they provide a natural setting for computing likelihoods of set-valued observations, in accordance with the random set philosophy.

When applied to samples generated by series of independent trials, under a generalisation of  stochastic independence, belief likelihoods factorise into simple products.
The resulting lower and upper likelihoods can be easily computed for series of Bernoulli trials,
and allows us to formulate a \emph{generalised logistic regression} framework, in which the mass values of individual trials are constrained to follow a logistic dependence on scalar parameters. The values of the parameters which optimise the lower and upper likelihoods induce a pair of `lower' and `upper' belief functions on the parameter space, whose interval effectively encodes the uncertainty associated with the amount of data at our disposal.
Every new observation, possibly in areas of the sample space not previously explored, is mapped to a pair of lower and upper logistic belief functions, which together provide lower and upper estimates for the belief values of each event. 

\subsection{Contributions}

The contributions of the paper are thus as follows:

(1) a belief likelihood function for repeated trials is defined, whenever the uncertainty on individual trials is assumed to be encoded by a belief measure;

(2) elegant factorisation properties are proven for events that are Cartesian products, whenever belief measures are combined by conjunctive rule, leading to the notions of lower and upper likelihoods;

(3) factorisation results are also provided in the case in which the dual, disjunctive combination is used to compute belief and plausibility likelihoods;

(4) analitical expressions of lower and upper likelihoods are provided for the case of Bernoulli trials;

(4) finally, a generalised logistic regression based on lower and upper likelihoods is formulated and analysed, as an alternative inference mechanism to generate belief functions from statistical data.

\subsection{Paper outline}

After reviewing in Section \ref{sec:logistic-regression} the logistic regression framework, we recall in Section \ref{sec:belief-functions} the necessary notions of the theory of belief functions. 
In Section \ref{sec:belief-likelihood} the belief likelihood function of repeated trials is defined. 
{In Section \ref{sec:belief-likelihood-conjunctive}, the belief likelihood of a series of binary trials is analysed in the conjunctive case. 
Factorisation results are shown 
which reduce upper and lower likelihoods of `sharp' samples to products of belief values of individual binary observations, and can be generalised to arbitrary Cartesian products of focal elements. 
In Section \ref{sec:belief-likelihood-disjunctive} an analysis of the belief likelihood function in the disjunctive case is conducted. 
General factorisation results holding for series of observations from arbitrary sample spaces are illustrated in Section \ref{sec:general}, while analytical expressions for the Bernoulli case are given in Section \ref{sec:bernoulli}.
}
Finally, a generalised logistic regression framework is outlined (Section \ref{sec:generalised-logistic}) in which the masses of the two outcomes are constrained to have a logistic dependency, and dual optimisation problems lead to a pair of lower and upper estimates for the belief measure of the outcomes.
Section \ref{sec:conclusions} concludes the paper and points at future work.

\section{Logistic regression} \label{sec:logistic-regression}

\emph{Logistic regression} allows us, given a sample $Y = \{Y_1,...,Y_n\}$, $X = \{x_1,...,x_n\}$ where $Y_i \in \{0,1\}$ is a binary outcome at time $i$ and $x_i$ is the corresponding observed measurement, to learn the parameters of a conditional probability relation between the two, of the form:
\begin{equation} \label{eq:logistic}
P(Y=1|x)={\frac {1}{1+e^{-(\beta _{0}+\beta _{1}x)}}},
\end{equation}
where $\beta_0$ and $\beta_1$ are two scalar parameters.
Given a new observation $x$, (\ref{eq:logistic}) delivers the probability of a positive outcome $Y=1$.
Logistic regression generalises deterministic linear regression, as it is a function of the linear combination $\beta _{0}+\beta _{1}x$.
The $n$ trials are assumed independent but not equally distributed, for $\pi_i = P(Y_i=1|x_i)$ varies with the time instant $i$ of collection.
The two scalar parameters $\beta_0,\beta_1$ in (\ref{eq:logistic}) are estimated by maximum likelihood of the sample. After denoting by
\begin{equation} \label{eq:logit}
\begin{array}{lll}
\pi_i & = & P(Y_i=1|x_i) = \displaystyle \frac{1}{1+e^{ - (\beta _{0}+\beta _{1}x_i) } }, 
\\ 
1 - \pi_i & = & P(Y_i=0|x_i) = \displaystyle \frac{e^{ - (\beta _{0}+\beta _{1} x_i) } }{1+e^{-(\beta _{0}+\beta _{1}x_i)} }
\end{array}
\end{equation}
the conditional probabilities of the two outcomes, the likelihood of the sample can be expressed as:
$
L(\beta|Y) = \prod_{i=1}^n \pi_i^{Y_i} (1-\pi_i)^{Y_i},
$
where $Y_i \in \{0,1\}$ and $\pi_i$ is a function of $\beta = [\beta_0,\beta_1]$. Maximising $L(\beta|Y)$ yields a conditional PDF $P(Y=1|x)$.
\\
Unfortunately, logistic regression shows clear limitations when the number of samples is insufficient or when there are too few positive outcomes (1s) \cite{King2001}. 
Moreover, inference by logistic regression tends to underestimate the probability of a positive outcome 
\cite{King2001}.

\section{Belief functions} \label{sec:belief-functions}

\subsection{Belief and plausibility measures} \label{sec:measures}

\begin{definition}\label{def:bpa}
A \emph{basic probability assignment} (BPA) \cite{Augustin96} over a finite domain $\Theta$ is a set function \cite{denneberg99interaction,dubois86set} $m : 2^\Theta\rightarrow[0,1]$ defined on the collection $2^\Theta$ of all subsets of $\Theta$ s.t.:
\[
m(\emptyset)=0, \; \sum_{A\subset\Theta} m(A)=1.
\]
\end{definition}
The quantity $m(A)$ is called the \emph{basic probability number} or `mass' \cite{kruse91tool,kruse91reasoning} assigned to $A$. 
The elements of the power set $2^\Theta$ associated with non-zero values of $m$ are called the \emph{focal elements} of $m$.
\begin{definition} \label{def:bel2}
The \emph{belief function} (BF) associated with a basic probability assignment $m : 2^\Theta\rightarrow[0,1]$ is the set function $Bel : 2^\Theta\rightarrow[0,1]$ defined as:
\begin{equation}\label{eq:belief} Bel(A) = \sum_{B\subseteq A} m(B). \end{equation}
\end{definition}
The domain $\Theta$ on which a belief function is defined is usually interpreted as the set of possible answers to a given problem, exactly one of which is the correct one. For each subset (`event') $A\subset \Theta$ the quantity $Bel(A)$ takes on the meaning of \emph{degree of belief} that the truth lies in $A$,
and represents the {total} belief committed to a set of possible outcomes $A$ by the available evidence $m$.

Another mathematical expression of the evidence generating a belief function $Bel$ 
is the \emph{upper probability} or \emph{plausibility} of an event $A$: $Pl(A) \doteq 1 - Bel(\bar{A})$,
as opposed to its \emph{lower probability} $Bel(A)$ \cite{cuzzolin10ida}. 
The {corresponding} \emph{plausibility function} $Pl : 2^\Theta \rightarrow [0,1]$ conveys the same information as $Bel$, and can be expressed as:
\[ Pl(A) = \sum_{B\cap A\neq \emptyset} m(B) \geq Bel(A). \]

\subsection{Evidence combination} \label{sec:combination}

The issue of combining the belief function representing our current knowledge state with a new one encoding the new evidence is central in belief theory. After an initial proposal by Dempster, several other aggregation operators have been proposed, based on different assumptions on the nature of the sources of evidence to combine.

\begin{definition} \label{def:dempster}
The \emph{orthogonal sum} or \emph{Dempster's combination} $Bel_1 \oplus Bel_2 : 2^\Theta \rightarrow [0,1]$ of two belief functions $Bel_1 : 2^\Theta \rightarrow [0,1]$, $Bel_2 : 2^\Theta \rightarrow [0,1]$ defined on the same domain $\Theta$ is the unique BF on $\Theta$ with as focal elements all the {non-empty} intersections of focal elements of $Bel_1$ and $Bel_2$, and basic probability assignment:
\begin{equation} \label{eq:dempster}
\displaystyle m_{\oplus}(A) = \frac{m_\cap(A)} {1- m_\cap(\emptyset)},
\end{equation}
where $m_i$ denotes the BPA of the input BF $Bel_i$, and:
$
m_\cap(A) = \sum_{B \cap C = A} m_1(B) m_2(C).
$
\end{definition}

Rather than normalising (as in (\ref{eq:dempster})), 
Smets' \emph{conjunctive rule} leaves the conflicting mass $m(\emptyset)$ with the empty set:
\begin{equation} \label{eq:combination-smets-conjunctive}
m_{\ocap}(A) = 
m_\cap (A) \quad \emptyset \subseteq A \subseteq \Theta, 
\end{equation}
and is thus applicable to `{unnormalised}' beliefs \cite{ubf}.

In Dempster's rule, consensus between two sources is expressed by the intersection of the supported events. When the \emph{union} 
is taken to express consensus we obtain the \emph{disjunctive} rule of combination \cite{Kramosil02probabilistic-analysis,YAMADA20081689}:
\begin{equation} \label{eq:combination-disjunctive}
m_{\ocup}(A) = \sum_{B \cup C = A} m_1(B) m_2(C),
\end{equation}
which yields more cautious inferences than conjunctive rules, by producing belief functions that are less `committed', i.e., have larger focal sets.
Under disjunctive combination: $Bel_1 \ocup Bel_2 (A) = Bel_1(A) \cdot Bel_2(A)$, input belief values are simply multiplied.

\subsection{Conditioning} \label{sec:conditioning}

Belief functions can also be conditioned, rather than combined, whenever we are presented hard evidence of the form `$A$ is true' 
\cite{Chateauneuf89,fagin91new,Jaffray92,gilboa93updating,denneberg94conditioning,yu94conditional,itoh95new}.

In particular,
Dempster's combination naturally induces a conditioning operator.
Given a conditioning event $A \subset \Theta$, the `logical' (or `categorical', in Smets' terminology) belief function $Bel_A$ such that $m(A)=1$ is combined via Dempster's rule with the a-priori belief function $Bel$. The resulting BF $Bel \oplus Bel_A$ is the {conditional belief function given $A$} \emph{a la Dempster}, denoted by $Bel_\oplus(A|B)$.

\subsection{Multivariate analysis} \label{sec:multivariate}

In many applications, we need to express uncertain information about a number of distinct variables (e.g., $X$ and $Y$) taking values in different domains ($\Theta_X$ and $\Theta_Y$, respectively). The reasoning process needs then to take place in the Cartesian product of the domains associated with each individual variable. 

Let then $\Theta_X$ and $\Theta_Y$ be two sample spaces associated with two distinct variables, and let $m^{XY}$ be a mass function on $\Theta_{XY} = \Theta_X \times \Theta_Y$.
The latter can be expressed in the coarser domain $\Theta_X$ by transferring each mass $m^{XY}(A)$ to the {projection} $A\downarrow \Theta_X$ of $A$ on $\Theta_X$.
We obtain a \emph{marginal} mass function on $\Theta_X$, denoted by: 
\[
m^{XY}_{\downarrow X}(B) \doteq \sum_{\{A \subseteq \Theta_{XY}, A\downarrow \Theta_X=B\}}  m^{XY}(A), \forall B \subseteq \Theta_X.
\]
Conversely, a mass function $m^X$ on $\Theta_X$ can be expressed in $\Theta_X\times \Theta_Y$  by transferring each mass $m^X(B)$ to the {cylindrical extension} $B^{\uparrow XY} \doteq B \times \Omega_Y$ of $B$.
The \emph{vacuous extension} of $m^X$ onto $\Theta_X\times \Theta_Y$ will then be:
\begin{equation} \label{eq:vacuous-extension}
m_X^{\uparrow XY}(A) \doteq
\begin{cases}
m^X(B) & \text{if } A=B \times \Omega_Y,\\
0 & \text{otherwise}.
\end{cases}
\end{equation}
The associated BF is denoted by $Bel_X^{\uparrow XY}$.

\section{Belief likelihood of repeated trials} \label{sec:belief-likelihood}

Let $Bel_{\X_i}(A|\theta)$, for $i=1,2,...,n$ be a parameterised family of belief functions on $\X_i$, the space of quantities that can be observed at time $i$, depending on a parameter $\theta \in \Theta$. A series of repeated trials then assumes values in $\X_1 \times \cdots \times \X_n$, whose elements are tuples of the form $\vec{x} = (x_1,...,x_n) \in \X_1 \times \cdots \times \X_n$.
We call such tuples `sharp' samples, as opposed to arbitrary subsets $A \subset \X_1 \times \cdots \times \X_n$ of the space of trials. Note that we are not assuming the trials to be equally distributed at this stage, nor we assume that they come from the same sample space.

\begin{definition} \label{def:belief-likelihood}
The belief likelihood function $Bel_{\X_1 \times \cdots \times \X_n} : 2^{\X_1 \times \cdots \times \X_n} \rightarrow [0,1]$ of a series of repeated trials is defined as:
\begin{equation} \label{eq:belief-likelihood}
Bel_{\X_1 \times \cdots \times \X_n}(A|\theta) \doteq Bel_{\X_1}^{\uparrow \times_i \X_i} \odot \cdots \odot Bel_{\X_n}^{\uparrow \times_i \X_i} (A|\theta),
\end{equation}
where $ Bel_{\X_j}^{\uparrow \times_i \X_i}$ is the vacuous extension (\ref{eq:vacuous-extension}) of $Bel_{\X_j}$ to the Cartesian product $\X_1 \times \cdots \times \X_n$ where the observed tuples live, and $\odot$ is an arbitrary combination rule.
\end{definition}
In particular, when the subset $A$ reduces to a sharp sample, $A = \{\vec{x}\}$, we can define the following generalisations of the notion of likelihood. 
\begin{definition} \label{def:lower-upper-likelihoods}
We call the quantities 
\begin{equation} \label{eq:lower-upper-likelihoods}
\begin{array}{l}
\underline{L}(\vec{x}) \doteq Bel_{\X_1 \times \cdots \times \X_n}(\{(x_1,...,x_n)\}|\theta),
\\ \\
\overline{L}(\vec{x}) \doteq Pl_{\X_1 \times \cdots \times \X_n}(\{(x_1,...,x_n)\}|\theta) 
\end{array}
\end{equation}
\emph{lower likelihood} and \emph{upper likelihood}, respectively, of $A = \{ \vec{x} \} = \{(x_1,...,x_n)\}$.
\end{definition}

\section{Binary trials: the conjunctive case} \label{sec:belief-likelihood-conjunctive}

Belief likelihoods factorise into simple products, whenever conjuctive combination is employed (as a generalisation of classical stochastic independence) in Definition \ref{def:belief-likelihood}, and trials with binary outcomes are considered.

\subsection{Focal elements of the belief likelihood} \label{sec:focal-elements-conjunctive}

Let us first analyse the case $n=2$. We seek the Dempster's sum $Bel_{\X_1} \oplus Bel_{\X_2}$, where $\X_1 = \X_2 = \{T,F\}$.
\\
Figure \ref{fig:casen2} is a diagram of all the intersections of focal elements of the two input BF on  $\X_1 \times \X_2$.
\begin{figure}[ht!]
\begin{center}
\includegraphics[width = 0.32\textwidth]{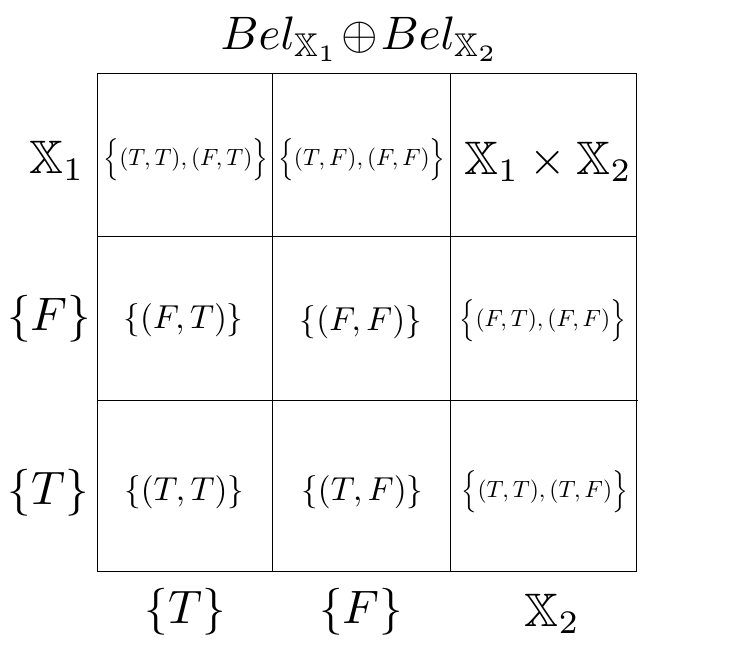}
\end{center}
\caption{Graphical representation of Dempster's combination $Bel_{\X_1} \oplus Bel_{\X_2}$ on $\X_1 \times \X_2$.} \label{fig:casen2}
\end{figure}
There are $9 = 3^2$ distinct, non-empty intersections, which correspond to the focal elements of $Bel_{\X_1} \oplus Bel_{\X_2}$. According to Equation (\ref{eq:dempster}), the mass of focal element $A_1 \times A_2$, $A_1 \subseteq \X_1$, $A_2 \subseteq \X_2$, is then:
\begin{equation} \label{eq:lemma-n2}
m_{Bel_{\X_1} \oplus Bel_{\X_2}} (A_1 \times A_2) = m_{{\X_1}} (A_1) \cdot m_{{\X_2}} (A_2).
\end{equation}
Note that the result holds when using the conjunctive rule $\ocap$ as well (\ref{eq:combination-smets-conjunctive}), for none of the intersections is empty, hence no normalisation is required. Nothing is assumed about the mass assignment of $Bel_{\X_1}$ and $Bel_{\X_2}$. 

We can now prove the following Lemma.
\begin{lemma} \label{lem:factorisation-conjunctive}
For any $n \in \mathbb{Z}$ the belief function $Bel_{\X_1} \oplus \cdots \oplus Bel_{\X_n}$, where $\X_i = \X = \{T,F\}$, has $3^n$ focal elements, namely all possible Cartesian products $A_1\times ... \times A_n$ of $n$ non-empty subsets $A_i$ of $\X$, with BPA:
\[
m_{Bel_{\X_1} \oplus \cdots \oplus Bel_{\X_n}} (A_1\times ... \times A_n) = \prod_{i=1}^n m_{\X_i}(A_i).
\]
\end{lemma}
\begin{proof}
The proof is by induction.
The thesis was shown to be true for $n=2$ in Equation (\ref{eq:lemma-n2}).
In the induction step, we assume that the thesis is true for $n$, and prove it for $n+1$.
If $Bel_{\X_1} \oplus \cdots \oplus Bel_{\X_{n}}$, defined on $\X_1 \times \cdots \times \X_{n}$, has as focal elements the $n$-products $A_1\times ... \times A_n$ with $A_i \in \big \{ \{T\}, \{F\}, \X \big \}$ for all $i$, its vacuous extension to $\X_1 \times \cdots \times \X_n \times \X_{n+1}$ will have as focal elements the
$n+1$-products of the form: $A_1\times ... \times A_n \times \X_{n+1}$,
with $A_i \in \big \{ \{T\}, \{F\}, \X \big \}$ for all $i$.

The belief function $Bel_{\X_{n+1}}$ is defined on $\X_{n+1} = \X$, with three focal elements: $\{T\}$, $\{F\}$ and $\X = \{T,F\}$. Its vacuous extension to $\X_1 \times \cdots \times \X_n \times \X_{n+1}$ thus has the following three focal elements:
$\X_1 \times \cdots \times \X_n \times \{T\}$, 
$\X_1 \times \cdots \times \X_n \times \{F\}$ and 
$\X_1 \times \cdots \times \X_n \times \X_{n+1}$.
\\
When computing $(Bel_{\X_1} \oplus \cdots \oplus Bel_{\X_{n}}) \oplus Bel_{\X_{n+1}}$ on the common refinement $\X_1 \times \cdots \times \X_n \times \X_{n+1}$ we need to compute the intersection of their focal elements, namely:
\[
\begin{array}{c}
\big ( A_1\times ... \times A_n \times \X_{n+1} \big ) \cap \big ( \X_1 \times \cdots \times \X_n \times A_{n+1} \big )
\\
=
A_1\times ... \times A_n \times A_{n+1}
\end{array}
\]
for all $A_{i} \subseteq \X_{i}$, $i=1,...,n+1$. All such intersections are distinct for distinct focal elements of the two belief functions to combine, and there are no empty intersection. By Dempster's rule (\ref{eq:dempster}) their mass is equal to the product of the original masses, i.e.:
\[
\begin{array}{c}
m_{Bel_{\X_1} \oplus \cdots \oplus Bel_{\X_{n+1}}}(A_1\times ... \times A_n \times A_{n+1}) =
\\
m_{Bel_{\X_1} \oplus \cdots \oplus Bel_{\X_n}}(A_1\times ... \times A_n) \cdot m_{Bel_{\X_{n+1}}}(A_{n+1}).
\end{array}
\]
Since we assumed that the factorisation holds for $n$, the thesis easily follows. 
\end{proof}

As no normalisation is involved in the combination $Bel_{\X_1} \oplus \cdots \oplus Bel_{\X_n}$, Dempster's rule coincides with the conjunctive rule and Lemma \ref{lem:factorisation-conjunctive} holds for $\ocap$ as well.

\subsection{Factorisation for `sharp' tuples} \label{sec:factorisation-conjunctive-sharp}

The following becomes then a simple corollary.
\begin{theorem} \label{the:belief-likelihood-decomposition-ocap}
When using either $\ocap$ or $\oplus$ as a combination rule in the definition of belief likelihood function, the following decomposition holds for tuples $(x_1,...,x_n)$, $x_i \in \X_i$, which are the singletons elements of $\X_1 \times \cdots \times \X_n$, with $\X_1 = ... = \X_n = \{T,F\}$:
\begin{equation} \label{eq:belief-likelihood-decomposition-ocap}
Bel_{\X_1 \times \cdots \times \X_n}(\{(x_1,...,x_n)\}|\theta) = 
\prod_{i=1}^n Bel_{\X_i} (\{x_i\}|\theta),
\end{equation}
\end{theorem}
\begin{proof}
For the singleton elements of $\X_1 \times \cdots \times \X_n$,
since $\{(x_1,...,x_n)\} = \{x_1\} \times ... \times \{x_n\}$, Equation (\ref{eq:belief-likelihood-decomposition-ocap}) becomes: 
$Bel_{\X_1 \times \cdots \times \X_n}(\{(x_1,...,x_n)\})
= m_{Bel_{\X_1} \oplus \cdots \oplus Bel_{\X_n}}(\{ (x_1,...,x_n) \}) 
= \prod_{i=1}^n m_{\X_i}(\{x_i \}) = \prod_{i=1}^n Bel_{\X_i} (\{x_i\})$,
where the mass factorisation follows from Lemma \ref{lem:factorisation-conjunctive},
as on singletons mass and belief values coincide. 
\end{proof}

There is evidence to support the following as well.
\begin{conjecture} \label{con:plausibility-likelihood-decomposition-ocap}
When using either \ocap$\hspace{0.3mm}$ or $\oplus$ as a combination rule in the definition of belief likelihood function, the following decomposition holds for the associated plausibility values on tuples $(x_1,...,x_n)$, $x_i \in \X_i$, which are the singletons elements of $\X_1 \times \cdots \times \X_n$, with $\X_1 = ... = \X_n = \{T,F\}$:
\begin{equation} \label{eq:plausibility-likelihood-decomposition-ocap}
\begin{array}{c}
\displaystyle
Pl_{\X_1 \times \cdots \times \X_n}(\{(x_1,...,x_n)\}|\theta) = \prod_{i=1}^n Pl_{\X_i} (\{x_i\}|\theta).
\end{array}
\end{equation}
\end{conjecture}

Indeed we can write: $Pl_{\X_1 \times \cdots \times \X_n}(\{(x_1,...,x_n)\}) =$
\begin{equation} \label{eq:plausibilities-singletons}
\begin{array}{l}
\displaystyle
= 1 - Bel_{\X_1 \times \cdots \times \X_n}(\{(x_1,...,x_n)\}^c) = \\
\displaystyle
= 1 - \sum_{B \subseteq \{(x_1,...,x_n)\}^c} m_{Bel_{\X_1} \oplus \cdots \oplus Bel_{\X_n}}(B).
\end{array}
\end{equation}
By Lemma \ref{lem:factorisation-conjunctive} all the subsets $B$ with non-zero mass are Cartesian products of the form $A_1 \times ... \times A_n$, $\emptyset \neq A_i \subseteq \X_i$.
We then need understand the nature of the focal elements of $Bel_{\X_1 \times \cdots \times \X_n}$ which are subsets of an arbitrary singleton complement $\{(x_1,...,x_n)\}^c$.

For binary spaces $\X_i = \X = \{T,F\}$, by definition of Cartesian product, each such $B = A_1 \times ... \times A_n \subseteq \{(x_1,...,x_n)\}^c$ is obtained by replacing a number $1 \leq k \leq n$ of components of the tuple $(x_1,...,x_n) = \{x_1\} \times ... \times \{x_n\}$ with a different subset of $\X_i$ (either $\{{x}_i\}^c = \X_i \setminus \{x_i\}$ or $\X_i$). There are $\binom{n}{k}$ such sets of $k$ components in a list of $n$.
Of these $k$ components, in general $1 \leq m \leq k$ will be replaced by $\{{x}_i\}^c$, while the other $1 \leq k-m < k$ will be replaced by $\X_i$. Note that not all $k$ components can be replaced by $\X_i$, since the resulting focal element would contain the tuple $\{(x_1,...,x_n)\} \in \X_1 \times ... \times \X_n$.

The following argument can be proved for $(x_1,...,x_n) = (T,...,T)$, under the additional assumption that $Bel_{\X_1} \cdots Bel_{\X_n}$ are equally distributed with $p \doteq Bel_{\X_i}(\{T\})$, $q \doteq Bel_{\X_i}(\{F\})$ and $r \doteq Bel_{\X_i}(\X_i)$.
\\
If this is the case, for fixed values of $m$ and $k$ all the resulting focal elements have the same mass value, namely: $p^{n-k} q^m r^{k-m}$, where $p \doteq Bel_{\X_i}(\{T\})$, $q \doteq Bel_{\X_i}(\{F\})$ and $r \doteq Bel_{\X_i}(\X_i)$. As there are exactly $\binom{k}{m}$ such focal elements, (\ref{eq:plausibilities-singletons}) can be written as:
\[
1 - \sum_{k=1}^n \binom{n}{k} \sum_{m=1}^k \binom{k}{m} p^{n-k} q^{m} r^{k-m}.
\]
which can be rewritten as:
\[
1 - \sum_{m=1}^n q^m \sum_{k=m}^n \binom{n}{k} \binom{k}{m} p^{n-k} r^{k-m}.
\]
A change of variable $l=n-k$, where $l=0$ when $k=n$, $l=n-m$ when $k=m$, allows us to write it as:
\[
1 - \sum_{m=1}^n q^m \sum_{l=0}^{n-m} \binom{n}{n-l} \binom{n-l}{m} p^{l} 
r^{(n-m)-l},
\]
since $k-m = n-l-m$, $k=n-l$. Now, as
\[
\binom{n}{n-l} \binom{n-l}{m} = \binom{n}{m} \binom{n-m}{l}
\]
we obtain:
$\displaystyle
1 - \sum_{m=1}^n q^m \binom{n}{m} \sum_{l=0}^{n-m} \binom{n-m}{l} p^{l} 
r^{(n-m)-l}.
$
By Newton's binomial, the latter is equal to
$\displaystyle
1 - \sum_{m=1}^n q^m \binom{n}{m} (1-q)^{n-m},
$
since $1 - q = p + r$.
Again, by Newton's binomial, we get:
\[
\begin{array}{l}
Pl_{\X_1 \times \cdots \times \X_n}(\{(T,...,T)\}) = 1 - [1 - (1-q)^n ] 
\\
\hspace{35mm} = \displaystyle (1 - q)^n = \prod_{i=1}^n Pl_{\X_i}(\{ T\}).
\end{array}
\]

\subsection{Factorisation for Cartesian products} \label{sec:factorisation-conjunctive-cartesian}

Decomposition (\ref{eq:belief-likelihood-decomposition-ocap}) is equivalent to what Smets calls \emph{conditional conjunctive independence} \cite{smets93belief}. 
In fact, for binary spaces factorisation (\ref{eq:belief-likelihood-decomposition-ocap}) generalises to all subsets $A \subseteq \X_1 \times \cdots \times \X_n$ of samples which are Cartesian products of subsets of $\X_1, ..., \X_n$, respectively: $A = A_1 \times \cdots \times A_n$, $A_i \subseteq \X_i$ for all $i$.

\begin{corollary} \label{cor:cci}
Whenever $A_i \subseteq \X_i = \{ T, F \}$, $i=1,...,n$, under conjunctive combination we have that:
\begin{equation} \label{eq:corollary}
Bel_{\X_1 \times \cdots \times \X_n}(A_1 \times \cdots \times A_n | \theta) = 
\prod_{i=1}^n Bel_{\X_i} (A_i | \theta).
\end{equation}
\end{corollary}
\begin{proof}
As by Lemma \ref{lem:factorisation-conjunctive} all the focal elements of $Bel_{\X_1} \ocap \cdots \ocap Bel_{\X_n}$ are Cartesian products of the form $B = B_1 \times \cdots \times B_n$, $B_i \subseteq \X_i$, it follows that $Bel_{\X_1 \times \cdots \times \X_n}(A_1 \times \cdots \times A_n | \theta)$ is equal to: 
\[
\sum_{B \subseteq A_1 \times \cdots \times A_n, B = B_1 \times \cdots \times B_n}
m_{\X_1}(B_1) \cdot ... \cdot m_{\X_n}(B_n).
\] 
But 
$
\{
B \subseteq A_1 \times \cdots \times A_n, B = B_1 \times \cdots \times B_n
\}
=
\{
B = B_1 \times \cdots \times B_n, B_i \subseteq A_i \forall i
\},
$
since if $B_j \not \subset A_j$ for some $j$ the resulting Cartesian product would not be a subset of $A_1 \times \cdots \times A_n$. 
Thus, $Bel_{\X_1 \times \cdots \times \X_n}(A_1 \times \cdots \times A_n | \theta)
=$
\begin{equation} \label{eq:aux-page6}
= \sum_{B = B_1 \times \cdots \times B_n, B_i \subseteq A_i \forall i}
m_{\X_1}(B_1) \cdot ... \cdot m_{\X_n}(B_n). 
\end{equation}
For all $A_i$'s, $i=i_1,...,i_m$ that are singletons of $\X_i$, necessarily $B_i = A_i$ and we can write (\ref{eq:aux-page6}) as:
\[
m(A_{i_1}) \cdot ... \cdot m(A_{i_m}) 
\sum_{B_j \subseteq A_j, j \neq i_1,...,i_m} \prod_{j \neq i_1,...,i_m} m_{\X_j}(B_j).
\]
If the frames are binary, $\X_i = \{T,F\}$, those $A_i$'s that are not singletons coincide with $\X_i$, so that we have:
\[
m(A_{i_1}) \cdot ... \cdot m(A_{i_m}) 
\sum_{B_j \subseteq \X_j, j \neq i_1,...,i_m} \prod_{j
} m_{\X_j}(B_j).
\]
The quantity $\displaystyle \sum_{B_j \subseteq \X_j, j \neq i_1,...,i_m} \prod_{j} m_{\X_j}(B_j)$ is, according to the definition of conjunctive combination, the sum of the masses of all the possible intersections of (cylindrical extensions of) focal elements of $Bel_{\X_j}$, $j \neq i_1,...,i_m$, thus they add up to 1.
In conclusion (forgetting the conditioning on $\theta$ in the derivation for sake of readability):
$
Bel_{\X_1 \times \cdots \times \X_n}(A_1 \times \cdots \times A_n) =  m(A_{i_1}) \cdot ... \cdot m(A_{i_m}) \cdot 1 \cdot ... \cdot 1 = Bel_{\X_{i_1}}(A_{i_1}) \cdot ... \cdot Bel_{\X_{i_m}}(A_{i_m}) \cdot Bel_{\X_{j_1}}(\X_{j_1}) \cdot ... \cdot Bel_{\X_{j_k}}(\X_{j_k}) =  
Bel_{\X_{i_1}}(A_{i_1}) \cdot ... \cdot Bel_{\X_{i_m}}(A_{i_m}) \cdot Bel_{\X_{j_1}}(A_{j_1}) \cdot ... \cdot Bel_{\X_{j_k}}(A_{j_k})
$
and we have (\ref{eq:corollary}).
\end{proof}

Corollary \ref{cor:cci} states that conditional conjunctive independence always holds for events that are Cartesian products, whenever the involved frames are binary. 

\section{Binary trials: the disjunctive case} \label{sec:belief-likelihood-disjunctive}

Similar factorisation results hold when using the (more cautious) disjunctive combination $\ocup$.

\subsection{Structure of the focal elements}

As in the conjunctive case, we first analyse the case $n=2$.
We seek the disjunctive combination $Bel_{\X_1} \ocup Bel_{\X_2}$, where each $Bel_{\X_i}$ has as focal elements $\{T\}$, $\{F\}$ and $\X_i$.
Figure \ref{fig:casen2-ocup} is a diagram of all the unions of focal elements of the two input BFs on their common refinement $\X_1 \times \X_2$.
\begin{figure}[ht!]
\begin{center}
\includegraphics[width = 0.32\textwidth]{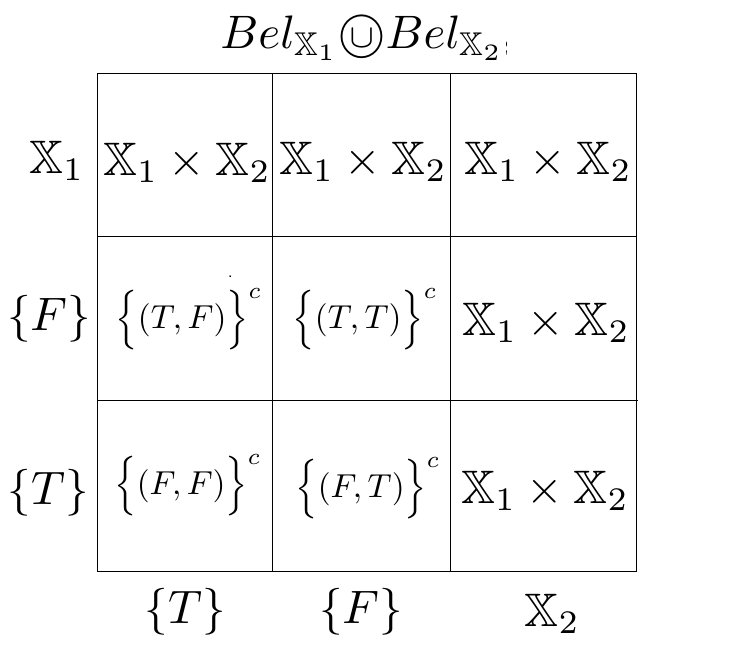}
\end{center}
\caption{Graphical representation of the disjunctive combination $Bel_{\X_1} \ocup Bel_{\X_2}$ on $\X_1 \times \X_2$.} \label{fig:casen2-ocup}
\end{figure}
There are $5 = 2^2 +1$ distinct such unions, the focal elements of $Bel_{\X_1} \ocup Bel_{\X_2}$, with masses:
\[
\begin{array}{l}
m (\{(x_i,x_j)\}^c) =
m_{{\X_1}} (\{x_i\}^c) \cdot m_{{\X_2}} (\{x_j\}^c),
\\
m (\X_1 \times \X_2) =
\displaystyle
1 - \sum_{i,j} m_{{\X_1}} (\{x_i\}^c) \cdot m_{{\X_2}} (\{x_j\}^c).
\end{array}
\]
We can now prove the following Lemma.
\begin{lemma} \label{lem:factorisation-disjunctive}
The belief function $Bel_{\X_1} \ocup \cdots \ocup Bel_{\X_n}$, where $\X_i = \X = \{T,F\}$, has $2^n + 1$ focal elements, namely all the complements of 
the $n$-tuples $\vec{x} = (x_1, ... , x_n)$ of singleton elements $x_i \in \X_i$, with BPA:
\begin{equation} \label{eq:mass-disjunctive}
\begin{array}{l} 
m_{Bel_{\X_1} \ocup \cdots \ocup Bel_{\X_n}}(\{ (x_1,...,x_n) \}^c) =
\\
\hspace{15mm}
= m_{\X_1}(\{x_1\}^c) \cdot \cdots \cdot m_{\X_n}(\{x_n\}^c),
\end{array}
\end{equation}
plus the Cartesian product $\X_1 \times \cdots \times \X_n$ itself, with mass value given by normalisation.
\end{lemma}
\begin{proof}
The proof is by induction. The case $n=2$ was proven above. In the induction step, we assume that the thesis is true for $n$, namely that the focal elements of $Bel_{\X_1} \ocup \cdots \ocup Bel_{\X_n}$ have the form:
\begin{equation} \label{eq:fe}
A = \{ (x_1,...,x_n) \}^c = \{ (x'_1,...,x'_n) | \exists i : \{x'_i\} = \{x_i\}^c \},
\end{equation}
where $x_i \in \X_i = \X$. We need to prove it true for $n+1$.
\\
The vacuous extension of (\ref{eq:fe}) has trivially the form:
\[
A' = \{ (x'_1,...,x'_n, x_{n+1}) |
\exists i 
: \{x'_i\} = \{x_i\}^c,  x_{n+1} \in \X \}.
\]
Note that only $2=|\X|$ singletons of $\X_1 \times \cdots \times \X_{n+1}$ are \emph{not} in $A'$, for any given tuple $(x_1,...,x_n)$.
\\
The vacuous extension to $\X_1 \times \cdots \times \X_{n+1}$ of a focal element $B = \{x_{n+1}\}$ of $Bel_{\X_{n+1}}$ is instead:
\[
B' = \{ (y_1, \cdots, y_n ,x_{n+1}) | y_i \in \X \; \forall i=1,...,n \}.
\]
Now, all the elements of $B'$, except for $(x_1,...,x_n,x_{n+1})$, are also elements of $A'$.
Hence, the union $A' \cup B'$ reduces to the union of $A'$ and $(x_1,...,x_n,x_{n+1})$.
The only singleton element of $\X_1 \times \cdots \times \X_{n+1}$ not in $A' \cup B'$ is therefore $(x_1,...,x_n, x'_{n+1})$, $\{ x'_{n+1} \} = \{ x_{n+1} \}^c$, for it is neither in $A'$ nor in $B'$.
All such unions are distinct. Thus, by definition of $\ocup$, their mass is $m(\{ (x_1,...,x_n) \}^c) \cdot m(\{x_{n+1}\}^c)$ which by inductive hypothesis is equal to (\ref{eq:mass-disjunctive}).
Unions involving either $\X_{n+1}$ or $\X_1 \times \cdots \times \X_n$ are equal to $\X_1 \times \cdots \times \X_{n+1}$ by the property of the union operator.
\end{proof}

\subsection{Factorisation} 

\begin{theorem} \label{the:belief-likelihood-decomposition-ocup}
In the hypotheses of Lemma \ref{lem:factorisation-disjunctive}, 
when using disjunctive combination \ocup \hspace{0.3mm} in the definition of belief likelihood function, the following decomposition holds: 
\begin{equation} \label{eq:likelihood-decomposition-ocup-belief}
Bel_{\X_1 \times \cdots \times \X_n}(\{(x_1,...,x_n)\}^c | \theta) = \prod_{i=1}^n Bel_{\X_i} (\{x_i\}^c | \theta).
\end{equation}
\end{theorem}
\begin{proof}
As $\{(x_1,...,x_n)\}^c$ contains only itself as a focal element:
\[
Bel_{\X_1 \times \cdots \times \X_n}(\{(x_1,...,x_n)\}^c | \theta) = m(\{(x_1,...,x_n)\}^c | \theta).
\]
By Lemma \ref{lem:factorisation-disjunctive} the latter becomes
\[
\begin{array}{lll}
& &  \displaystyle
Bel_{\X_1 \times \cdots \times \X_n}(\{(x_1,...,x_n)\}^c | \theta)
\\
& = & \displaystyle
\prod_{i=1}^n m_{\X_i} (\{x_i\}^c | \theta) = \prod_{i=1}^n Bel_{\X_i} (\{x_i\}^c | \theta),
\end{array}
\]
as $\{x_i\}^c$ is a singleton element of $\X_i$, and we have (\ref{eq:likelihood-decomposition-ocup-belief}).
\end{proof}

Note that $Pl_{\X_1 \times \cdots \times \X_n}(\{(x_1,...,x_n)\}^c | \theta) = 1$ for all tuples $(x_1,...,x_n)$, as the set $\{(x_1,...,x_n)\}^c$ has non-empty intersection with all the focal elements of $Bel_{\X_1} \ocup \cdots \ocup Bel_{\X_n}$.

\section{General factorisation results} \label{sec:general}

The argument of Lemma \ref{lem:factorisation-conjunctive} is in fact valid for the conjunctive combination of belief functions defined on an arbitrary collection $\X_1, ...,\X_n$ of finite spaces.

\begin{theorem} \label{the:factorisation-conjunctive}
For any $n \in \mathbb{Z}$ the belief function $Bel_{\X_1} \oplus \cdots \oplus Bel_{\X_n}$, where $\X_1,...,\X_n$ are finite spaces, has as focal elements all the Cartesian products $A_1\times ... \times A_n$ of $n$ focal elements $A_1 \subseteq \X_1, ..., A_n \subseteq \X_n$, with BPA:
\[
m_{Bel_{\X_1} \oplus \cdots \oplus Bel_{\X_n}} (A_1\times ... \times A_n) = \prod_{i=1}^n m_{\X_i}(A_i).
\]
\end{theorem}
The proof is similar to that of Lemma \ref{lem:factorisation-conjunctive}, and is omitted for lack of space. It follows that:

\begin{corollary} \label{cor:belief-likelihood-decomposition-ocap}
When using either \ocap \hspace{0.3mm} or $\oplus$ as a combination rule in the definition of belief likelihood function, the following decomposition holds for tuples $(x_1,...,x_n)$, $x_i \in \X_i$, which are the singletons elements of $\X_1 \times \cdots \times \X_n$, with $\X_1, ... , \X_n$ any finite frames of discernment:
\begin{equation} \label{eq:belief-likelihood-decomposition-ocap-general}
\begin{array}{c}
\displaystyle
Bel_{\X_1 \times \cdots \times \X_n}(\{(x_1,...,x_n)\}|\theta) = 
\prod_{i=1}^n Bel_{\X_i} (\{x_i\}|\theta).
\end{array}
\end{equation}
\end{corollary}
\begin{proof}
For the singleton elements of $\X_1 \times \cdots \times \X_n$, since $\{(x_1,...,x_n)\} = \{x_1\} \times ... \times \{x_n\}$, Equation (\ref{eq:belief-likelihood-decomposition-ocap-general}) becomes: $Bel_{\X_1 \times \cdots \times \X_n}(\{(x_1,...,x_n)\}) =$
$m_{Bel_{\X_1} \oplus \cdots \oplus Bel_{\X_n}}(\{ (x_1,...,x_n) \}) = \prod_{i=1}^n m_{\X_i}(\{x_i \}) =$ $\prod_{i=1}^n Bel_{\X_i} (\{x_i\})$,
where the mass factorisation follows from Theorem \ref{the:factorisation-conjunctive},
as on singletons mass and belief values coincide. 
\end{proof}

\section{Lower and upper likelihoods of Bernoulli trials} \label{sec:bernoulli}

In the case of Bernoulli trials, where not only there is a single binary sample space, $\X_i = \X = \{T,F\}$ and conditional independence holds, but the random variables are assumed equally distributed, the conventional likelihood reads as $p^k (1-p)^{n-k}$, where $p = P(T)$, $k$ is the number of successes ($T$) and $n$ the total number of trials.
\\
Let us then compute the lower likelihood function for a series of Bernoulli trials, under the assumption that all the BFs $Bel_{\X_i} = Bel_{\X}$, $i=1,...,n$, coincide (the analogous of equidistribution), with $Bel_\X$  parameterised by $p=m(\{T\})$, $q = m(\{F\})$ (where, this time, $p +q \leq 1$). 
\begin{corollary}
Under the above assumptions, the lower and upper likelihoods of the sample $\vec{x} = (x_1,...,x_n)$ are:
\begin{equation} \label{eq:lower-upper-likelihoods-bernoulli}
\begin{array}{l}
\underline{L}(\vec{x}) = \displaystyle \prod_{i=1}^n Bel_\X(\{x_i\})
= p^k q^{n-k};
\\
\overline{L}(\vec{x}) =
\displaystyle \prod_{i=1}^n Pl_\X(\{x_i\}) = (1-q)^k (1-p)^{n-k}.
\end{array}
\end{equation}
\end{corollary}
The above decomposition for $\overline{L}(\vec{x})$, in particular, is valid under the assumption that Conjecture \ref{con:plausibility-likelihood-decomposition-ocap} holds, at least for Bernoulli trials, as the evidence seems to suggest.
\\
After normalisation, these can be seen as probability distribution functions (PDFs) over the (belief) space $\mathcal{B}$ of all belief functions definable on $\X$ \cite{cuzzolin08smcc,cuzzolin14lap,cuzzolin18springer}.

Having observes a series of trials, $\vec{x} = ( x_1,...,x_n )$, $x_i \in \{T,F\}$, one may then seek the belief function on $\X = \{T,F\}$ which best describes the observed sample, i.e., the optimal values of the two parameters $p$ and $q$. 

\begin{figure}[ht!] 
\begin{tabular}{c}
\hspace{-5mm}\includegraphics[width = 0.5\textwidth]{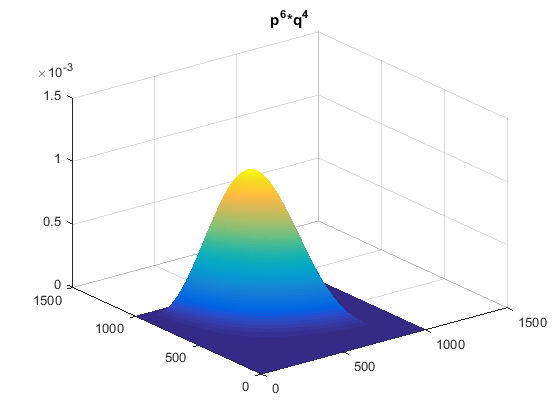}
\\
\hspace{-5mm}
\includegraphics[width = 0.5\textwidth]{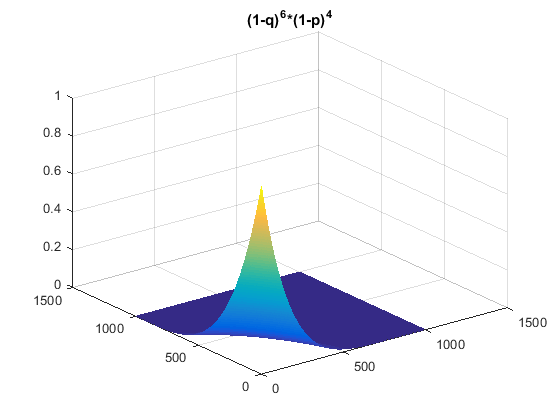}
\end{tabular}
\caption{Lower (top) and upper (bottom) likelihood functions plotted over the space of belief functions on $\X = \{ T,F \}$, parameterised by $p = m(T)$ ($X$ axis) and $q = m(F)$ ($Y$ axis), for the case of $k=6$ successes over $n=10$ trials.} \label{fig:lower-upper-likelihoods}
\end{figure} 

Figure \ref{fig:lower-upper-likelihoods} plots the both lower and upper likelihoods (\ref{eq:lower-upper-likelihoods-bernoulli}) for the case of $k=6$ successes over $n=10$ trials.
\\
Both subsume the traditional likelihood $p^k(1-p)^{n-k}$, as their section for $p+q =1$, although this is particularly visible for the lower likelihood (top). In particular, the maximum of the lower likelihood is the traditional ML estimate $p = k/n$, $q = 1 - p$.
This makes sense, for the lower likelihood is highest for the most committed belief functions (i.e., for probability measures).
The upper likelihood (bottom) has a unique maximum in $p=q=0$: this is the vacuous belief function on $\{T,F\}$, with $m(\{T,F\}) = 1$.

The interval of belief functions joining $\max \overline{L}$ with $\max \underline{L}$ is the set of belief functions such that $\frac{p}{q} = \frac{k}{n-k}$, i.e., those which \emph{preserve the ratio between the observed empirical counts}.

\section{Generalising logistic regression} \label{sec:generalised-logistic}

Based on Theorem \ref{the:belief-likelihood-decomposition-ocap} and Conjecture \ref{con:plausibility-likelihood-decomposition-ocap}, 
we can also generalise logistic regression (Section \ref{sec:logistic-regression}) to a belief function setting by replacing the conditional probability $(\pi_i, 1-\pi_i)$ on $\X = \{T,F\}$ with a belief function ($p_i = m(\{T\})$, $q_i = m(\{F\})$, $p_i + q_i \leq 1$) on $2^\X$.
\\
Note that, just as in a traditional logistic regression setting (Section \ref{sec:logistic-regression}), the belief functions $Bel_i$ associated with different input values $x_i$ are \emph{not} equally distributed.

After writing $T=1, F=0$, the lower and upper likelihoods can be expressed as:
\[
\begin{array}{lll}
\underline{L}(\beta|Y) & = & \displaystyle \prod_{i=1}^n p_i^{Y_i} q_i^{1-Y_i}, 
\\ 
\overline{L}(\beta|Y) & = & \displaystyle \prod_{i=1}^n (1 - q_i)^{Y_i} (1 - p_i)^{1-Y_i}.
\end{array}
\]
The question becomes how to generalise the logit link between observations $x$ and outputs $y$, in order to seek an analytical mapping between observations and belief functions over a binary frame. Just assuming:
\begin{equation}\label{eq:logit3}
p_i = P(Y_i=1|x_i) = \displaystyle \frac{1}{1+e^{ - (\beta _{0}+\beta _{1}x_i) } }, 
\end{equation}
as in the classical contraint (\ref{eq:logit}), does not yield any analytical dependency for $q_i$.
To address this issue we can, for instance, add a parameter $\beta_2$ such that the following relationship holds:
\begin{equation}\label{eq:logit2}
q_i = m(Y_i=0|x_i) = \beta_2 \frac{e^{-(\beta _{0}+\beta _{1}x_i)}}{1+e^{-(\beta _{0}+\beta _{1}x_i)}}.
\end{equation}
We can then seek lower and upper optimal estimates for the parameter vector $\beta = [\beta_0,\beta_1,\beta_2]'$:
\begin{equation} \label{eq:generalised-logistic}
\arg\max_{\beta} \underline{L} \mapsto \underline{\beta}_0, \underline{\beta}_1, \underline{\beta}_2, 
\quad 
\arg\max_{\beta} \overline{L} \mapsto \overline{\beta}_0, \overline{\beta}_1, \overline{\beta}_2.
\end{equation}
Plugging these optimal parameters into (\ref{eq:logit3}), (\ref{eq:logit2}) will then yield a lower and an upper family of conditional belief functions given $x$ (i.e., an interval of belief functions):
\[
Bel_\X (.|\underline{\beta},x), \quad Bel_\X (.|\overline{\beta},x).
\]
Given any new test observation $x'$, this generalised logistic regression method will then output a pair of lower and upper belief functions on $\X = \{T,F\}$, as opposed to a sharp probability value as in the classical framework. As each belief function itself provides a lower and an upper probability for each event, both the lower and the upper regressed BFs will provide an interval for the probability $P(T)$ of success, whose width will reflect the uncertainty encoded by the training set of sample series.

\subsection{Optimisation}

The problems (\ref{eq:generalised-logistic}) are both constrained optimisation ones (contrarily to the classical case, where $q_i=1-p_i$ and the optimisation problem is unconstrained). 
\\
Indeed, the parameter vector $\beta$ must be such that:
\[
0 \leq p_i + q_i \leq 1 \quad \forall i=1,...,n.
\]
Fortunately the number of constraints can be reduced by noticing that, if $p_i$ and $q_i$ have the analytical forms (\ref{eq:logit3}) and (\ref{eq:logit2}), respectively, then $p_i + q_i \leq 1$ for all $i$ whenever $\beta_2 \leq 1$.
\\
The objective function can also be simplified by taking the logarithm. In the lower likelihood case we get\footnote{Derivations are omitted due to lack of space.}:
\[
\begin{array}{lll}
\log \underline{L}(\beta|Y) & = & \displaystyle \sum_{i=1}^n \Big \{ -\log (1+e^{-(\beta_0+\beta_1 x_i)})
\\
& & + (1 - Y_i) \big  [ \log \beta_2 - (\beta_0 + \beta_1 x_i) \big ] \Big \}.
\end{array}
\]

We then need to analyse the Karush-Kuhn-Tucker (KKT) necessary conditions for the optimality of the solution of a nonlinear optimisation problem $\arg\max_x f(x)$
subject to differentiable constraints: $g_i(x)\leq 0$ $i=0,...,n$.
\\
If $x^{*}$ is a local optimum, under some regularity conditions then there exist constants $\mu _{i}$, $(i=0,\ldots ,n)$, called \emph{KKT multipliers}, such that the following conditions hold:
\begin{enumerate}
\item
$\nabla f(x^{*}) - \sum _{i=1}^{m}\mu _{i}\nabla g_{i}(x^{*}) = 0$ (\emph{Stationarity});
\item
\emph{Primal feasibility}: 
$g_{i}(x^{*})\leq 0$ for all $i=0,\ldots ,n$;
\item
\emph{Dual feasibility}:
$\mu _{i}\geq 0$ for all $i=0,\ldots ,n$;
\item
\emph{Complementary slackness}: 
$\mu _{i}g_{i}(x^{*})=0$ for all $i$.
\end{enumerate}
In our case the constraints are:
\[
\begin{array}{lll}
\beta_1 \leq 1  & \equiv & g_0 = \beta_2 - 1 \leq 0;
\\
p_i + q_i \geq 0 & \equiv & g_i = - \beta_2 - e^{\beta_0 + \beta_1 x_i} \leq 0
\end{array}
\]
$i=1,...,n$, and the Lagrangian becomes:
\[
\Lambda(\beta) = \log \underline{L}(\beta) + \mu_0 (\beta_2-1) - \sum_{i=1}^n \mu_i (\beta_2 + e^{\beta_0 + \beta_1 x_i}).
\]
The stationarity conditions thus read as $\nabla\Lambda(\beta) = 0$, i.e.:
\begin{equation} \label{eq:kkt-stationarity}
\left \{
\begin{array}{l}
\displaystyle
\sum_{i=1}^n \Big [ (1-p_i) - (1-Y_i) - \mu_i e^{\beta_0+\beta_1 x_i} \Big ] = 0,
\\
\displaystyle
\sum_{i=1}^n \Big [ (1-p_i) - (1-Y_i) - \mu_i e^{\beta_0+\beta_1 x_i} \Big ] x_i = 0,
\\
\displaystyle
\sum_{i=1}^n \left (\frac{1 - Y_i}{\beta_2} - \mu_i \right ) + \mu_0 = 0,
\end{array}
\right .
\end{equation}
where, as usual, $p_i$ is as in (\ref{eq:logit3}).
\\
Complementary slackness reads instead as follows:
\begin{equation} \label{eq:kkt-slackness}
\left \{
\begin{array}{l}
\mu_0 (\beta_2 - 1) = 0,
\\
\mu_i (\beta_2 + e^{\beta_0+\beta_1 x_i}) = 0, \quad i=1,...,n.
\end{array}
\right .
\end{equation}

Standard gradient descent methods can be applied to the above systems of equations to get the optimal parameters of our generalised logistic regressor.
Similar calculations hold for the upper likelihood problem. A multi-objective optimisation setting in which the lower likelihood is minimised as the upper likelihood is maximised can be envisaged, to generate the most cautious interval of estimates.

\section{Conclusions} \label{sec:conclusions}

In this paper, 
stimulated by the inability of logistic regression to encode uncertainty induced by scarcity of samples in certain areas of the sample space, and inference mechanisms for belief measures which take the classical likelihood function at face value, we defined a belief likelihood function for repeated trials, as an inference methodology both more in line with the random set philosophy of inherently set-valued observations, and capable of allowing a promising generalisation of logistic regression.
We analysed the factorisation properties of the belief likelihood function when either
conjunctive or disjunctive combination is applied, in particular for the case of binary spaces, and computed the lower and upper likelihoods of series of Bernoulli trials.
Eventually, we proposed a generalised logistic regression framework which leads to a pair of dual constrained optimisation problems, which can be solved by standard methods. 

A number of interesting research lines lie ahead. Among others, a systematic comparison with other approaches to belief function inference,  
or the computation of belief and plausibility likelihood for other major combination rules.
As far as generalised logistic regression is concerned, different parameterisations of the belief functions involved need to be explored. A comprehensive testing of the robustness of the generalised framework in standard estimation problems, including rare event analysis, needs to be conducted to validate this new procedure. 
The method, unlike traditional ones, can naturally cope with missing data (represented by vacuous observations $x_i = \X$), therefore providing a robust framework for logistic regression which can deal with incomplete series of observations.


\providecommand{\bysame}{\leavevmode\hbox to3em{\hrulefill}\thinspace}
\providecommand{\MR}{\relax\ifhmode\unskip\space\fi MR }
\providecommand{\MRhref}[2]{%
  \href{http://www.ams.org/mathscinet-getitem?mr=#1}{#2}
}
\providecommand{\href}[2]{#2}

\end{document}